\documentclass[a4paper,12pt]{article}
\usepackage[danish,english,]{babel}
\usepackage[T1]{fontenc}
\usepackage[utf8]{inputenc}
\usepackage{lmodern}

\usepackage{amsthm,amsmath,amsfonts}

\usepackage{enumitem}
\usepackage{mathtools}
\usepackage{bbm}
\usepackage{prodint}
\usepackage{natbib}

\newcounter{maincounter}
\newtheorem{theorem}[maincounter]{Theorem}
\newtheorem{prop}[maincounter]{Proposition}
\newtheorem{lemma}[maincounter]{Lemma}
\theoremstyle{definition}

\newcommand{\given}{\mathbin{|}}
\newcommand{\R}{\mathbb{R}}
\newcommand{\N}{\mathbb{N}}
\renewcommand{\d}{\hspace{1pt} \mathrm{d}}
\newcommand{\F}{\mathcal{F}}
\renewcommand{\P}{\operatorname{P}}
\newcommand{\partition}[1]{\mathcal{#1}}
\newcommand{\E}{\operatorname{E}}

\newcommand{\J}{\mathcal{J}}
\newcommand{\from}{\colon}
\newcommand{\banach}[1]{\mathbf{#1}}

\renewcommand{\c}{^{\mathsf{c}}}

\newcommand{\indic}[1]{\mathbf{1}(#1)}
\newcommand{\one}{\mathbbm{1}}

\title{State occupation probabilities \\ in non-Markov models}
\author{
Morten Overgaard \\
\small Department of Public Health, Aarhus University, \\
\small Bartholins All\'{e} 2, DK-8000 Aarhus C, Denmark \\
\small moov@ph.au.dk
}
\date{May 31, 2019}

\begin{document}

\maketitle

\begin{abstract}
The consistency of the Aalen--Johansen-derived estimator of state occupation probabilities in non-Markov multi-state settings is studied and established via a new route. 
This new route is based on interval functions and relies on a close connection between additive and multiplicative transforms of interval functions, which is established.
Under certain assumptions, the consistency follows from explicit expressions of the additive and multiplicative transforms related to the transition probabilities as interval functions, which are obtained, in combination with certain censoring and positivity assumptions. \\[12pt]
\end{abstract}

\section{Introduction}

The Aalen--Johansen estimator of transition probabilities in multi-state models and the derived estimator of state occupation probabilities are known to be consistent when the Markov property holds for the multi-state process, which may be subject to independent censoring.
A result by \cite{Datta2001} is that the estimator of state occupation probabilities derived from the Aalen--Johansen estimator remains valid under standard assumptions even in the non-Markov case.
Some steps of the argument seem to rely on martingale properties of certain processes.
Although these processes are martingales in a Markov setting, it is not clear to this author that they retain the necessary martingale properties generally in a non-Markov setting.
In any case, it is of interest to establish the same result without the use of martingale arguments.

In this paper, the consistency of the Aalen--Johansen-derived estimate of state occupation probabilities is established by appealing to a simple identity for the state occupation probability and results on additive and multiplicative transforms of interval functions that are established.
This approach offers further insights into why the consistency continues to hold in the non-Markov case.

\section{The multi-state setting}
\label{sec:multi_state}

Consider a càdlàg multi-state process $U$ with state space $\{1, \dots, d\}$ and time parameter space $[0, \infty)$. The state occupation probabilities are given by the row vector $p(s)$ with entries $p_j(s) = \P(U(s) = j)$. With the definition $P_{jk}(s,t) = \P(U(t) = k \given U(s) = j)$, a transition matrix is defined by $P(s,t) = \{P_{jk}(s,t)\}$.
The conditional probability is $P_{jk}(s,t) = \P(U(t) = k, U(s) = j)/p_j(s)$ when $p_j(s) > 0$ and taken to be $P_{jk}(s,t) = \indic{j=k}$ otherwise.
We define a cumulative transition hazard by $\Lambda_{jk}(s,t) = \int_s^t p_j(u-)^{-1} F_{jk}(\d u)$ for $k \neq j$ where $F_{jk}(s) = \E(\# \{u \in (0,s] \given U(u) = k, U(u-) = j\})$ is the expected number of direct transitions from $j$ to $k$ up to time $s$.
A cumulative transition hazard matrix $\Lambda$ is then defined by having $\Lambda_{jk}$ as the $(j,k)$ entry when $k \neq j$ and $-\sum_{k \neq j} \Lambda_{jk}$ as the $(j,j)$ entry.

With full information on independent replications of the multi-state process $U$, the natural estimator of the state occupation probability $p(t)$ is an average of the state occupation indicators over replications. If $n$ independent replications $U_1, \dots, U_n$ of $U$ are observed, take the estimate $\hat p_n(t)$ which has entries $\hat p_{n,j}(t) = n^{-1} \sum_{i=1}^n \indic{U_i(t) = j}$.
Estimation is complicated by censoring of the multi-state process $U$. Consider a multi-state process $X$ with state space $\{0, \dots, d\}$ fulfilling $X(t) = U(t)$ when $X(t) \neq 0$. Then $X$ can be considered a censored version of $U$ with $X(t) = 0$ denoting that $U(t)$ is unobserved.
The state 0 may or may not be absorbing for $X$.
Generally, but perhaps especially when 0 is not absorbing for $X$, the term filtering rather than censoring of $U(t)$ for the case $X(t) = 0$ may be more in line with the usual terminology, for instance with the terminology from \cite{Andersen1993}.

Consider $n$ independent replications $X_1, \dots, X_n$ of $X$ as the observed information. Let $N_{jk}\c(s) = \#\{u \in (0,s] \given X(u) = k, X(u-) = j\}$ denote the number of transitions from state $j$ to $k$ and let $F_{jk}\c(s) = \E(N_{jk}\c(s))$ denote the mean. For replication $i$, we let $N_{i,jk}\c(s) = \#\{u \in (0,s] \given X_i(u) = k, X_i(u-) = j\}$, and an empirical mean is defined by $\hat F_{n,jk}\c(s) = n^{-1} \sum_{i=1}^n N_{i,jk}\c(s)$.
Similarly, we use $Y_j\c(s) = \indic{X(s) = j}$ for state occupation with expectation $p_j\c(s) = \E(Y_j\c(s))$. For replication $i$, we let $Y_{i,j}\c(s) = \indic{X_i(s) = j}$, and an empirical mean is defined by $\hat p_{n,j}\c (s) = n^{-1} \sum_{i=1}^n Y_{i,j}\c(s)$.
The Nelson--Aalen estimate of $\Lambda_{jk}(s,t)$ is
\begin{equation}
\hat \Lambda_{n,jk}(s,t) = \int_s^t \frac{1}{\hat p_{n,j}\c(u-)} \hat F_{n,jk}\c(\d u).
\end{equation}
With $\hat \Lambda_{n,jj}(s,t) = - \sum_{k \neq j} \hat \Lambda_{n,jk}(s,t)$, a matrix $\hat \Lambda_n(s,t) = \{\hat \Lambda_{n,jk}(s,t)\}$ can be defined.
Based on this, the Aalen--Johansen estimate of the matrix $P(s,t) = \{P_{jk}(s,t)\}$ is, as defined in \cite{AalenJohansen1978},
\begin{equation}
\hat P_{n}(s,t) = \Prodi_s^t(\one+\hat \Lambda_{n}(\d u)),
\end{equation}
where $\one$ denotes the identity matrix. The derived estimate of $p(t)$ is 
\begin{equation}
\hat p_{n}(t) = \hat p_{n}(0) \hat P_{n}(0,t)
\end{equation}
for $t > 0$ and $\hat p_{n,j}(0) = \hat p_{n,j}\c(0)/ \sum_{k=1}^d \hat p_{n,k}\c(0)$.

It should not be surprising that $\hat \Lambda_{n,jk}(s,t)$ tends to converge to $\Lambda_{jk}\c(s,t) := \int_s^t p_j\c(u-)^{-1} F_{jk}\c(\d u)$, the observable transition hazard, rather than $\Lambda_{jk}(s,t)$ as desired. 
In order for this approach to work, we make the assumption $\Lambda_{jk}\c(s,t) = \Lambda_{jk}(s,t)$ for all $s, t$ with $s \leq t$ for all $j \in \{1, \dots, d\}$ and $k \in \{1, \dots, d\} \setminus \{j\}$. This is a weak version of an independent censoring assumption and is equivalent to assuming
\begin{equation}
\begin{aligned}
    &\P(X(s-) = j, X(s) = k \given U(s-) = j, U(s) = k) \\
    & \quad = \P(X(s-) = j \given U(s-) = j)
\end{aligned}
\end{equation}
for $F_{jk}$-almost all $s$ of interest for all $j \in \{1, \dots, d\}$ and $k \in \{1, \dots, d\}\setminus\{j\}$. This can be called the status-independent observation assumption since, for fixed $j \in \{1, \dots, d\}$, it states that among the statuses of transitioning from $j$ to $k$ at time $s$ and being in state $j$ immediately before time $s$, the probability of observing such status does not depend on the status. 
This term is along the lines of \cite{overgaard2019assumption} and the equivalence mentioned can be established using the techniques of that paper.
Also, in order for $\hat p_n(0)$ to be a consistent estimate of $p(0)$, the assumption $\sum_{k=1}^d p_k\c(0) > 0$ and $p_j(0) = p_j\c(0)/\sum_{k=1}^d p_k\c(0)$ for $j=1, \dots, d$, or equivalently that $\P(X(0) = j \given U(0) = j)$, the probability of observing the initial state given that the initial state is $j$, for $j$ with $p_j(0) > 0$ is positive and does not depend on $j$, is appropriate.

\begin{prop} \label{prop:AJ_consistency}
For given $t \in (0, \infty)$, assume $\E(N_{jk}\c(t)^2) < \infty$ and that $p_j\c(u-) \geq \varepsilon$ for some $\varepsilon > 0$ for $F_{jk}$-almost all $u \in (0,t]$ for all $j \in \{1, \dots, d\}$ and $k  \in \{1, \dots, d\}\setminus \{j\}$. Also, assume that $\sum_{k=1}^d p_k\c(0) > 0$ and $p_j(0) = p_j\c(0)/\sum_{k=1}^d p_k\c(0)$ and $\Lambda_{jk}\c(s,u) = \Lambda_{jk}(s,u)$ for all $s \leq u \leq t$ for all $j \in \{1, \dots, d\}$ and $k \in \{1, \dots, d\} \setminus \{j\}$. Then 
\begin{equation}
\hat p_{n}(s) \to p(0) \Prodi_0^s(\one + \Lambda(\d u)).
\end{equation}
in probability as $n \to \infty$ uniformly for $s \in (0,t]$.
\end{prop}
\begin{proof}
Since the Markov property is not assumed to hold, the usual martingale arguments are not expected to work. In particular, the process $N_{jk}\c(s) - \int_0^s Y_j\c(u-) \Lambda_{jk}\c(\d u)$ is not expected to be a martingale since $\Lambda_{jk}\c(s)$ is not expected to take all past information into account. 
The result can be proven by taking the functional approach of \cite{Glidden2002} in this setting. Or the result can be proven by taking a functional approach based on $p$-variation for a $p \in (1,2)$ as laid out in \cite{overgaard2019counting} since the underlying functionals are continuous in a $p$-variation setting and since $\|\hat F_{n,jk}\c - F_{jk}\c\|_{[p]} \to 0$ in probability for all $j \in \{1, \dots, d\}$ and $k \in \{1, \dots, d\}\setminus\{j\}$ for such $p$ under the assumptions where $\| \cdot \|_{[p]}$ is the $p$-variation norm. This yields the convergence in $p$-variation norm on $(0,t]$ and so in particular uniformly on $(0,t]$.
Either approach can be used to study the asymptotic properties of the estimator in more detail as is done in \cite{Glidden2002}. 
\end{proof}

We have, by the definitions, for any $s \leq t$,
\begin{equation}
    p(t) = p(s) P(s,t)
\end{equation}
and, by iterating, this leads to
\begin{equation} \label{eq:p_t_iterated}
p(t) = p(s) P(t_0, t_1) \cdots P(t_{m-1}, t_m) = p(s) \prod_{i=1}^m P(t_{i-1}, t_i)
\end{equation}
for any choice of time points $s = t_0 < \dots < t_m = t$, as also pointed out by \cite{Aalen2001}.
On the basis of~\eqref{eq:p_t_iterated} and Proposition~\ref{prop:AJ_consistency}, the remaining task of this paper is to argue that the limit over refinements $\prodi_{(0,t]} \d P := \lim \prod_{i=1}^m P(t_{i-1}, t_i)$ exists and equals $\prodi_0^t(\one+ \Lambda(\d u))$.
If this holds, by taking the limit in \eqref{eq:p_t_iterated},
\begin{equation}
p(t) = p(0) \lim \prod_{i=1}^m P(t_{i-1}, t_i) = p(0) \Prodi_0^t(\one + \Lambda(\d s)),
\end{equation}
which is consistently estimated by the Aalen--Johansen estimators of the state occupation probabilities under some assumptions according to Proposition~\ref{prop:AJ_consistency}, establishing the desired result. 
As a consequence of Theorem~\ref{theorem:Lambda_vs_P} below the limit $\lim \prod_{i=1}^m P(t_{i-1}, t_i)$ exists and equals $\prodi_0^t(\one+ \Lambda(\d u))$ as desired under an upper continuity requirement and a bounded variation requirement on the $P_{jk}$s.

On a side note, the identity \eqref{eq:p_t_iterated} for the empirical distribution with time points $t_1, \dots, t_m$ at transition times also explains why the Aalen--Johansen estimators of state occupation probabilities are simply the observed proportions in the uncensored case as also established in section IV.4.1.4 of \cite{Andersen1993}. 

\section{Interval functions and their transforms}

The concept of interval functions, as known partly from \cite{Gill1990} but especially from \cite{Dudley2010}, will be at the core of the argument presented here. Consider an interval $J \subseteq \R$ and the set of all subintervals of $J$, denoted $\J$. An interval function is a function defined on such a $\J$. The interval functions we will consider here map into $\R$ or, more generally, into the vector space of $d \times d$ matrices, $\banach{M}$, which will be equipped with the maximum norm and the standard matrix multiplication. We let $\one \in \banach{M}$ denote the identity matrix. 
With $\one$ as the identity element, $\banach{M}$ is a unital Banach algebra, satisfying $\|xy\| \leq \|x\| \|y\|$ for elements $x,y \in \banach{M}$, and $\banach{M}$ can be considered any general unital Banach algebra in the following.

We use the notation $A < B$ for intervals $A, B \in \J$ if $a < b$ for any choices of $a \in A$ and $b \in B$. 
Two types of interval functions are important here: 
\begin{itemize}
    \item An interval function $\mu \from \J \to \banach{M}$ is said to be \emph{additive} if
    \begin{equation}
        \mu(A \cup B) = \mu(A) + \mu(B)
    \end{equation}
    for any $A, B \in \J$ such that $A \cap B = \emptyset$ and $A \cup B \in \J$.
    \item An interval function $\mu \from \J \to \banach{M}$ is said to be \emph{multiplicative} if
    \begin{equation}
        \mu(A \cup B) = \mu(A)\mu(B)
    \end{equation}
    for any $A < B \in \J$ such that $A \cup B \in \J$.
\end{itemize}
Since $\banach{M}$ is not generally commutative, the order of multiplication matters in the definition of a multiplicative interval function and here the stated definition is used in line with \cite{Gill1990} but at odds with the definition preferred by \cite{Dudley2010}.

A partition $\partition{A}$ of an interval $A \in \J$ is a finite set $\partition{A} = \{A_i\}_{i=1}^m$ of subintervals $A_1 < A_2 < \dots < A_m$ of $A$ such that $\cup_{i=1}^m A_i = A$. 
A variation concept for interval functions is defined by
\begin{equation}
\| \mu \|_{(1)} = \sup_{\partition{A}} \sum_{A \in \partition{A}} \| \mu(A) \|
\end{equation}
where the supremum is over partitions $\partition{A}$ of $J$. An interval function $\mu$ is then of bounded variation when $\| \mu \|_{(1)} < \infty$.
In this case, a real-valued interval function is obtained by $\|\mu\|_{(1)}(A) := \sup_{\partition{B}} \sum_{B \in \partition{B}} \| \mu(B) \|$, where the supremum is over partitions $\partition{B}$ of~$A$.

If both $\partition{A}$ and $\partition{B}$ are partitions of an interval $A$, $\partition{B}$ is called a refinement of $\partition{A}$ if any $B \in \partition{B}$ is a subinterval of an interval in $\partition{A}$.
For an interval $A$, $|A|$ is the length of the interval, and for a partition $\partition{A} = \{A_i\}_{i=1}^m$ the mesh is $|\partition{A}| = \max_{i} |A_i|$.
Consider a function $S$ which associates any partition, $\partition{A}$, of an interval $A$ with an element $S(\partition{A}) \in \banach{M}$. 
Two notions of a limit will be of interest:
\begin{itemize}
    \item If $V \in \banach{M}$ is such that for each $\varepsilon > 0$ a partition $\partition{A}$ exists such that, for any refinement $\partition{B}$ of $\partition{A}$, $\|S(\partition{B}) - V \| < \varepsilon$ then we say that $V$ is the limit of $S$ over refinements, which is denoted by $V = \lim_{\partition{A}} S(\partition{A})$.
    
    \item If $V \in \banach{M}$ is such that for each $\varepsilon > 0$ a $\delta > 0$ exists such that, for any partition $\partition{A}$ with $|\partition{A}| < \delta$, $\|S(\partition{A}) - V \| < \varepsilon$ then $V$ is the limit of $S$ in mesh, which is denoted by $V = \lim_{|\partition{A}| \to 0} S(\partition{A})$. 
\end{itemize}

It is worth noting that if $V$ is the limit of $S$ in mesh then $V$ is also the limit of $S$ over refinements.
Another useful fact is that $V$ is a limit of $S$ in mesh, $V = \lim_{|\partition{A}| \to 0} S(\partition{A})$, if and only if $\lim_{n \to \infty} \|S(\partition{A}_n) - V\| = 0$ for any sequence of partitions $(\partition{A}_n)$ with $|\partition{A}_n| \to 0$ as $n \to \infty$.

Examples of $S$ as considered above are $S(\mu; \partition{A}) = \sum_{A \in \partition{A}} \mu(A)$ and $S(\mu; \partition{A}) = \prod_{A \in \partition{A}} \mu(A) = \mu(A_1) \mu(A_2) \cdots \mu(A_m)$ for partitions $\partition{A} = \{A_i\}_{i=1}^m$ of $J$ with $A_1 < A_2 < \cdots < A_m$ for an interval function $\mu \from \J \to \banach{M}$. 
Limits of these $S$ lead to what will be called additive and multiplicative transforms of $\mu$.

\begin{itemize}
    \item If, for a given interval function $\mu \from \J \to \banach{M}$, for any $A \in \J$, the limit over refinements of $A$, $\int_A \d \mu := \lim_{\partition{B}}\sum_{B \in \partition{B}} \mu(B)$ exists, the interval function $A \mapsto \int_A \d \mu$ is called the additive transform of $\mu$.
    
    \item If, for a given interval function $\mu \from \J \to \banach{M}$, for any $A \in \J$, the limit over refinements of $A$, $\prodi_A \d \mu := \lim_{\partition{B}}\prod_{B \in \partition{B}}\mu(B)$ exists, the interval function $A \mapsto \prodi_A \d \mu$ is called the multiplicative transform of $\mu$.
\end{itemize}

Either of the transforms will be unique when it exists.
Clearly, the additive transform, when it exists, is an additive interval function and the multiplicative transform, when it exists, is a multiplicative interval function.

At this stage it is worth noting that what we are ultimately looking for is to establish the existence of a multiplicative transform of $P$ as an interval function with an expression as a product integral. Also, the product integral $\prodi_s^t(\one + \Lambda(\d u))$ corresponds to the evaluation $\prodi_{(s,t]} \d (\one + \Lambda)$ of the multiplicative transform of $\one + \Lambda$ seen as an interval function. 

In the following, somewhat stricter versions of the additive and multiplicative transforms will be useful. 
\begin{itemize}
    \item A strict additive transform of an interval function $\mu$ is an additive interval function $\check \mu$ such that for any $\varepsilon > 0$ a partition $\partition{A}$ of $J$ exists such that
    \begin{equation}
    \sum_{B \in \partition{B}} \| \mu(B) - \check \mu(B) \| < \varepsilon
    \end{equation}
    for any refinement $\partition{B}$ of $\partition{A}$.
    
    \item A strict multiplicative transform of an interval function $\mu$ is a multiplicative interval function $\hat \mu$ such that for any $\varepsilon > 0$ a partition $\partition{A}$ of $J$ exists such that
    \begin{equation}
    \sum_{B \in \partition{B}} \| \mu(B) - \hat \mu(B) \| < \varepsilon
    \end{equation}
    for any refinement $\partition{B}$ of $\partition{A}$.
\end{itemize}
By the triangle inequality, it can be seen that a strict additive transform is, in fact, an additive transform $\check \mu(A) = \lim_{\partition{B}} \sum_{B \in \partition{B}} \mu(B) = \int_A \d \mu$, where the limit is over refinements of partitions $\partition{B}$ of $A$, for any $A \in \J$.
Similarly, Theorem~9.34 of \cite{Dudley2010} establishes that a strict multiplicative transform is a multiplicative transform $\hat \mu(A) = \lim_{\partition{B}} \prod_{B \in \partition{B}} \mu(B) = \prodi_A \d \mu$, where the limit is over refinements of partitions $\partition{B}$ of $A$, for any $A \in \J$ under the assumption that $\sup_{A \in \J} \|\mu(A)\| < \infty$, which is implied by $\|\mu - \one\|_{(1)} < \infty$, for instance. What is here called a strict multiplicative transform of $\mu$ corresponds to a multiplicative transform of $A \mapsto \mu(A) - \one$ in the terminology of \cite{Dudley2010}.

An important result is the following.
\begin{theorem} \label{theorem:additive_multiplicative}
Consider an interval function $\mu \from \J \to \banach{M}$ of bounded variation. Then $\mu$ has a strict additive transform, $\check \mu$, if and only if $\nu := \one + \mu$ has a strict multiplicative transform, $\hat \nu$. When this happens, $\hat \nu$ is also the strict multiplicative transform of $\one + \check \mu$ and $\check \mu$ is also the strict additive transform of $\hat \nu - \one$. 
\end{theorem}
\begin{proof}
Assume that $\mu$ has a strict additive transform $\check \mu$. Since, for any partition  $\partition{A}$, $\sum_{A \in \partition{A}} \|\check \mu(A)\| \leq \sum_{B \in \partition{B}} \| \mu(B) \| + \sum_{B \in \partition{B}} \|\mu(B) - \check \mu(B)\|$ for any refinement $\partition{B}$ of $\partition{A}$, we see from the properties of the strict additive transform that $\|\check \mu\|_{(1)} \leq \|\mu \|_{(1)}$.
Using the arguments of Section~2 of \cite{Gill1990} or of Chapter~9 of \cite{Dudley2010}, it can be established that if $\eta$ is an additive interval function of bounded variation then the strict multiplicative transform $\hat \xi$ of $\xi := \one + \eta$ exists, and similarly, if $\xi$ is a multiplicative interval function and $\eta := \xi - \one$ is of bounded variation then the strict additive transform $\check \eta$ of $\eta$ exists.
According to this result, $\one + \check \mu$ has a strict multiplicative transform.
Let $\hat \nu$ denote the strict multiplicative transform of $\one + \check \mu$. Then, for any partition $\partition{B}$ of $J$,
\begin{equation}
\sum_{B \in \partition{B}} \| \one + \mu(B) - \hat \nu(B)\| \leq \sum_{B \in \partition{B}} \| \mu(B) - \check \mu(B)\| + \sum_{B \in \partition{B}} \|\one + \check \mu(B) - \hat \nu(B)\|.
\end{equation}
For any $\varepsilon > 0$, we can find a partition $\partition{A}$ such that either term on the right-hand side is smaller than $\varepsilon/2$ whenever $\partition{B}$ is a refinement of $\partition{A}$. In particular, this shows that the strict multiplicative transform of $\nu = \one + \mu$ exists and corresponds to $\hat \nu$, the strict multiplicative transform of $\one + \check \mu$. The other implication is shown in a similar fashion, where it is important to note that if the multiplicative transform $\hat \nu$ of $\one + \mu$ exists then, for any partition~$\partition{A}$,
\begin{equation}
\sum_{A \in \partition{A}} \| \hat \nu(A) - \one \| \leq \exp(\sum_{B \in \partition{B}}\|\hat \nu(B) - \one \|) \sum_{B \in \partition{B}}\| \hat \nu(B) - \one\|
\end{equation} 
for any refinement $\partition{B}$ of $\partition{A}$ since $\hat \nu(A) - \one = \sum_{i=1}^m \prod_{j=1}^{i-1} \hat \nu(B_j)(\hat \nu(B_i) - \one)$ by multiplicativity if $\{B_i\}_{i=1}^m$ is a partition of $A$ and since $\|\hat \nu(B)\| \leq 1 + \|\hat \nu(B) - \one\| \leq \exp(\|\hat \nu(B) - \one\|)$ such that $\|\hat \nu - \one\|_{(1)} \leq \exp(\|\mu\|_{(1)}) \|\mu\|_{(1)} < \infty$. 
\end{proof}

An interval function $\mu$ is said to be upper continuous if, for all $A \in \J$, $\mu(A_n) \to \mu(A)$ for any $(A_n) \subseteq \J$ with $A_n \downarrow A$ as $n \to \infty$. 
According to Proposition~2.6 of \cite{Dudley2010}, an additive interval function $\mu$ is upper continuous when and only when $\mu(A_n) \to 0$ for any $(A_n) \subseteq \J$ with $A_n \downarrow \emptyset$, which is called upper continuity at $\emptyset$. For a strict additive transform $\check \mu$ of an interval function $\mu$, this is the case when $\mu$ is upper continuous at $\emptyset$.

A function $f \from J \to \R$ is said to be regulated if it has limits from the left as well as from the right everywhere where applicable, potentially including at $-\infty$ and $\infty$ if $J$ is unbounded. In particular, a regulated function is bounded and has at most a finite number of jumps larger than any fixed $\varepsilon > 0$.
A second important result is the following.
\begin{theorem} \label{theorem:integral}
Consider an interval function $\mu \from \J \to \R$ which is upper continuous at $\emptyset$ and has bounded variation and which has a strict additive transform, $\check \mu$. 
Consider also a regulated function $f \from J \to \R$. Define an interval function $\nu$ by $\nu(A) = f(u) \mu(A)$ when left end point $u$ of $A$ is in $A$ and by $\nu(A) = f(u+) \mu(A)$ when $u$ is not in $A$. 
Then $\nu$ has a strict additive transform, $\check \nu$, which is given by the Kolmogorov integral $\check \nu(A) = \int_A f \d\check \mu$.
\end{theorem}
\begin{proof}
Since $\check \mu$ will be additive, upper continuous and of bounded variation and $f$ is regulated, the Kolmogorov integral $\int_A f \d\check \mu$ exists as a consequence of Theorem~2.20 and Proposition~2.25 of \cite{Dudley2010}. The Kolmogorov integral satisfies $\|\int_A f \d\check \mu\| \leq \sup_{s \in A} |f(s)| \|\check \mu\|_{(1)}(A)$, where ${\| \cdot \|_{(1)}(A)}$ is the variation on $A$.
We will consider a partition $\partition{A}$ of $J$ with elements of the form $\{t_i\}$ and $(t_{i-1},t_i)$.
Such a partition is called a Young partition in \cite{Dudley2010}.
Since $f$ is regulated we can, according to Theorem~2.1 of \cite{Dudley2010}, find such a partition $\partition{A}$ such that the oscillation of $f$ on the interval $(t_{i-1}, t_i)$, $\sup_{s,u \in (t_{i-1},t_i)}\|f(s) - f(u)\|$, does not exceed a given $\varepsilon > 0$ for any $i$. Potentially by a refinement, we can take $\partition{A}$ such that also $\sum_{B \in \partition{B}} \|\mu(B) - \check \mu(B)\| < \varepsilon$ for any refinement $\partition{B}$ of $\partition{A}$ since $\check \mu$ is the strict additive transform of $\mu$.
Now, consider any refinement $\partition{B}$ of $\partition{A}$ and let $s_B$ denote any member of $B \in \partition{B}$ and, if $u$ is the left end point of $B \in \partition{B}$, let $y_B = f(u)$ if $u \in B$ and $y_B = f(u+)$ if $u \notin B$.
We then have the conclusion that
\begin{equation}
\begin{aligned}
\sum_{B \in \partition{B}} \| \nu(B) - \int_B f \d \check \mu \| &\leq \sum_{B \in \partition{B}} \| (y_B - f(s_B)) \mu(B) \| \\
&\phantom{{}=} + \sum_{B \in \partition{B}} \|f(s_B)(\mu(B) - \check \mu(B))\| \\
&\phantom{{}=} + \sum_{B \in \partition{B}} \| \int_{B} (f(s_B) - f) \d \check \mu\| \\
& \leq \varepsilon \|\mu\|_{(1)} + \varepsilon \|f\|_{\infty} + \varepsilon \|\check \mu\|_{(1)},
\end{aligned}
\end{equation} 
which can be made arbitrarily small by an appropriate choice of $\varepsilon$.
\end{proof}
Since $\check \mu$ is upper continuous and of bounded variation under the assumptions of Theorem~\ref{theorem:integral} and since a regulated function is bounded and Borel measurable, the Kolmogorov integral of the theorem also corresponds to the Lebesgue-Stieltjes integral.

The concept of a random interval function on probability space $(\Omega, \F, \P)$ can be introduced as a function $\mu \from \Omega \times \J \to \R$ such that $\omega \mapsto \mu(\omega; A)$ is $\F$-Borel measurable for all $A \in \J$. This concept will be useful in the following.

\section{Statement and proof of main result}
\label{sec:main_result}
Let us consider $J = (0, \tau]$ for some $\tau > 0$ and define the interval functions that are relevant in the mutli-state context.
We will consider $P_{jk}$ an interval function with definitions
\begin{align*}
    P_{jk}((s,t]) &= \P(U(t) = k \given U(s) = j), \\
    P_{jk}((s,t)) &= \P(U(t-) = k \given U(s) = j), \\
    P_{jk}([s,t)) &= \P(U(t-) = k \given U(s-) = j), \\
    P_{jk}([s,t]) &= \P(U(t) = k \given U(s-) = j),
\end{align*}
for $s \leq t$.
Here, we can again take $P_{jk}((s,t]) = P_{jk}((s,t)) = \indic{j=k}$ if $p_j(s) = 0$ and similarly $P_{jk}([s,t]) = P_{jk}([s,t)) = \indic{j=k}$ if $p_j(s-) = 0$.
We have $P_{jk}((s,t]) = P_{jk}(s,t)$ and  $P_{jk}((s,t) ) = P_{jk}(s,t-)$ whenever $p_j(s) > 0$, and $P_{jk}([s,t]) = P_{jk}(s-,t)$ and $P_{jk}([s,t)) = P_{jk}(s-,t-)$ whenever $p_j(s-) > 0$.
Similarly the matrix-valued $P$ can be considered an interval function with the interval function $P_{jk}$ as the $(j,k)$th entry.
Also, $\Lambda_{jk}$ as an interval function is given by $\Lambda_{jk}(A) = \int_A p_j(u-)^{-1} F_{jk}(\d u)$ for an interval $A$. This defines an additive interval function with values in $[0, \infty]$. 
As an interval function, $P$ is a multiplicative interval function when and only when the Chapman--Kolmogorov equation $P(s,t) = P(s,u) P(u,t)$ for $s \leq u \leq t$ holds. In the non-Markov setting we consider, this is not generally the case.

If we consider again the multi-state process $U$, then $M_{jk}$ with $M_{jk}((s,t]) = \indic{U(t) = k, U(s) = j}$, $M_{jk}((s,t)) = 1(X(t-) = k, X(s) = j)$, $M_{jk}([s,t]) = 1(X(t) = k, X(s-) = j)$, and $M_{jk}([s,t)) = 1(X(t-) = k, X(s-) = j)$ 
for $k \neq j$ is a random, upper continuous interval function.
The interval functions defined, for intervals $A$, by $N_{jk}(\omega; A) := \# \{u \in A \given U(\omega; u) = k, U(\omega; u-) = j\}$ for $k \neq j$ will be important. These interval functions are additive and upper continuous.
Since a change in state for $U$ involves at least one direct transition somewhere, we have $| M_{jk}(\omega; A) | \leq \sum_{h \neq j} N_{jh}(\omega; A)$ and so $\| M_{jk}(\omega; \cdot) \|_{(1)}(A) \leq \sum_{h \neq j} N_{jh}(\omega; A)$ for $k \neq j$ due to additivity of $N_{jh}(\omega; \cdot)$. 
The interval function $N_{jk}(\omega; \cdot)$ is also the candidate for the strict additive transform of $M_{jk}(\omega; \cdot)$. 
Since we consider $J = (0, \tau]$ for some $\tau > 0$, we have that for any sequence of partitions $(\partition{A}_n)$ with mesh converging to 0, $\lim_n \sum_{A \in \partition{A}_n} \| M_{jk}(\omega; A) - N_{jk}(\omega; A) \| = 0$ for almost all $\omega$ since $\partition{A}_n$ separates jumps when the mesh is sufficiently small. 
In particular, $N_{jk}(\omega; \cdot)$ will be the strict additive transform of $M_{jk}(\omega; \cdot)$ in this case, and $N_{jk}$ is a random interval function since, for any $A \in \J$, $N_{jk}(A)$ is the limit of $\F$-Borel measurable functions like $\sum_{B \in \partition{B}_n} M_{jk}(B)$ for partitions $\partition{B}_n$ of $A$.
We define interval functions by $Q_{jk}(A) = \E(M_{jk}(A))$ and by $F_{jk}(A) := \E(N_{jk}(A))$. Here, $Q_{jk}$ is upper continuous. As an interval function, $F_{jk}$ is additive and, at least if $F_{jk}(J) < \infty$, also upper continuous.

\begin{prop} \label{prop:F_jk_transform_Q_jk}
For a given $j$, assume $F_{jk}(J) < \infty$ for all $k \neq j$. Then $F_{jk}$ is the strict additive transform of $Q_{jk}$ for all $k \neq j$.
\end{prop}
\begin{proof}
For any given $k \neq j$ and any sequence of partitions $(\partition{A}_n)$ with ${|\partition{A}_n| \to 0}$ as $n \to \infty$,
\begin{equation}
\sum_{A \in \partition{A}_n} | Q_{jk}(A) - F_{jk}(A) | \leq \E \big( \sum_{A \in \partition{A}_n} | M_{jk}(A) - N_{jk}(A)| \big) \to 0
\end{equation}
for $n \to \infty$ by dominated convergence since $\sum_{A \in \partition{A}_n} | M_{jk}(A) - N_{jk}(A)| \leq 2 \sum_{h \neq j} N_{jh}(J)$, which is integrable under the assumption. In particular, $S(\partition{A}) := \sum_{A \in \partition{A}} | Q_{jk}(A) - F_{jk}(A) |$ has limit $\lim_{|\partition{A}| \to 0} S(\partition{A}) = 0$ in mesh and so over refinements, which is the requirement for $F_{jk}$ to be the strict additive transform of $Q_{jk}$.
\end{proof}
As a consequence of Proposition~\ref{prop:F_jk_transform_Q_jk} and the argument found in the proof of Theorem~\ref{theorem:additive_multiplicative}, we obtain $\|Q_{jk}\|_{(1)} \geq F_{jk}(J)$ when $F_{jk}(J) < \infty$, but this conclusion holds generally in the sense that $F_{jk}(J) = \infty$ implies $\|Q_{jk}\|_{(1)} = \infty$ which can be seen as a result of Fatou's lemma.
From the point-wise bound $\| M_{jk}(\omega; \cdot) \|_{(1)}(A) \leq \sum_{h \neq j} N_{jh}(\omega; A)$, we obtain $\|Q_{jk}\|_{(1)} \leq \sum_{h \neq j} F_{jh}(J)$ also.

The main result is given as follows. Recall that we are considering a bounded interval $J= (0, \tau]$. 
\begin{theorem} \label{theorem:Lambda_vs_P}
Assume $P - \one$ is upper continuous at $\emptyset$ and of bounded variation. 
Then $\Lambda$ is the strict additive transform of $P - \one$. 
\end{theorem}

\begin{proof}
The assumption implies that $\|P_{jk}\|_{(1)} < \infty$ for all $j$ and $k \neq j$. 
We consider now such a $j$ and $k \neq j$. 
Since $P_{jk}((s,t]) = Q_{jk}((s,t])/p_j(s) \geq Q_{jk}((s,t])$ when $p_j(s) > 0$ and $P_{jk}((s,t]) = 0 = Q_{jk}((s,t])$ otherwise and similarly for other types of intervals, we have $\infty > \|P_{jk}\|_{(1)} \geq \|Q_{jk}\|_{(1)} \geq F_{jk}(J)$. 
Split $J$ into $J_{j+}$ and $J_{j0}$ where $J_{j+} = \{s \in J : p_j(s) > 0 \textup{ and } p_j(s-) > 0 \}$ and $J_{j0} = \{s \in J : p_j(s) = 0 \textup{ or } p_j(s-) = 0 \}$ which are open and closed respectively relative to $J$.
For any interval of the type $[t_0,t_1] \subseteq J_{j+}$ an $\varepsilon > 0$ exists such that $p_j(u-) \geq \varepsilon$ and $p_j(u) \geq \varepsilon$ for all $u \in [t_0,t_1]$ by Lemma~\ref{lemma:pos_set} of the appendix. As a function on $[t_0, t_1]$, $s \mapsto p_j(s-)^{-1}$ is then regulated. With $\mu = Q_{jk}$ and $f(s) = p_j(s-)^{-1}$, Theorem~\ref{theorem:integral} now implies that $\Lambda_{jk}$ is the strict additive transform of $P_{jk}$ on $[t_0, t_1]$.
It is worth noting about $\Lambda_{jk}$ that additivity and non-negativity means that $\|\Lambda_{jk}\|_{(1)}(A) = \Lambda_{jk}(A)$ for any interval $A \in \J$. Also that $\Lambda_{jk}((s,t)) = 0 = P_{jk}((s,t))$ for $(s,t) \subseteq J_{j0}$, $\Lambda_{jk}([t,t]) = \P(X(t) = k \given X(t-) = j) = P_{jk}([t,t])$ generally, and $\Lambda_{jk}((s,t)) = \sup_{[u,v] \subseteq (s,t)} \Lambda_{jk}([u,v]) \leq \sup_{[u,v] \subseteq (s,t)} \|P_{jk}\|_{(1)}([u,v]) \leq \|P_{jk}\|_{(1)}((s,t))$ for $(s,t) \subseteq J_{j+}$ since $\Lambda_{jk}$ is the strict additive transform of $P_{jk}$ on any $[u,v] \subseteq (s,t)$ in this case and since $\Lambda_{jk}$ is upper continuous by definition.

Next, an interval partition of $J$ is considered. 
As an open set relative to $J$, $J_{j+}$ is the countable union of open intervals, open relative to $J$. 
According to Lemma~\ref{lemma:Lambda_extremes} of the appendix, if $[t_0, t_1) \subseteq J_{j+}$ and $t_1 \in J_{j0}$ then either $\Lambda_{jk}([t_0,t_1)) = \infty$ for some $k \neq j$ if $p_j(t_1-) = 0$ or $\sum_{k \neq j} \Lambda_{jk}([t_1,t_1]) = 1$ if $p_j(t_1-) > 0$. 
The first case cannot be encountered and the second case can only be encountered a finite number of times on $J$ since $\Lambda_{jk}$ is dominated by $\|P_{jk}\|_{(1)} < \infty$ for all $k \neq j$ on these types of intervals. This means that $J_{j+}$ is actually a union of finitely many open intervals and this implies the existence of a partition $J=U_1 \cup \{u_1\} \cup U_2 \cup \dots \cup U_m \cup \{u_m\}$ with open intervals $U_i \subseteq J_{j0}$ or $U_i \subseteq J_{j+}$ and with $u_m = \tau$.
Additionally, we necessarily have $p_j(u_i-) > 0$ for all $i$. The existence of such a partition also means that $\Lambda_{jk}(J) \leq \|P_{jk}\|_{(1)} < \infty$ can easily be established from the results above.

Following the proof of Proposition~3.50 of \cite{Dudley2010}, it can be proven that upper continuity of $P_{jk}$ and the assumption $\|P_{jk}\|_{(1)} < \infty$ lead to $A \mapsto \|P_{jk}\|_{(1)}(A)$ being upper continuous at $\emptyset$.
This means that for any $\varepsilon > 0$, we can find a $\delta > 0$ such that $\|P_{jk}\|_{(1)}(A) < \frac{\varepsilon}{4m}$ for all intervals $A$ among $(0, \delta)$ and $(u_i, u_i + \delta)$ for $i=1, \dots, m-1$. In particular, $\sum_{B \in \partition{B}} | P_{jk}(B) - \Lambda_{jk}(B)| < \frac{\varepsilon}{2m}$ for any partition $\partition{B}$ of such an interval $A$. 
Consider $[u_i+\delta, u_{i+1}) \subseteq U_{i+1} \subseteq J_{j+}$. Since $p_{j}(u_{i+1}-) > 0$, the argument of Lemma~\ref{lemma:pos_set} of the appendix leads to the existence of an $\tilde \varepsilon > 0$ such that $p_{j}(v) \geq \tilde \varepsilon$ for all $v \in [u_i + \delta, u_{i+1})$. Then, as seen above, $\Lambda_{jk}$ is the strict additive transform of $P_{jk}$ on $[u_i+\delta, u_{i+1})$. This is also trivially the case when $[u_i+\delta, u_{i+1}) \subseteq U_{i+1} \subseteq J_{j0}$ since both $P_{jk}$ and $\Lambda_{jk}$ are 0 on the open $U_{i+1}$.
So, for each $i$, find a partition $\partition{A}_i$ of $[u_i+\delta, u_{i+1})$ such that $\sum_{B \in \partition{B}} \| P_{jk}(B) - \Lambda_{jk}(B) \| < \frac{\varepsilon}{2m}$ for any refinement $\partition{B}$ of $\partition{A}_i$. 
Put together, this yields a partition $\partition{A} = \{(0, \delta)\} \cup \partition{A}_1 \cup \{ [u_1, u_1] \} \cup \{(u_1, u_1 + \delta)\} \cup \dots \cup \partition{A}_m \cup \{ [\tau, \tau]\}$ such that for any refinement $\partition{B}$ of $\partition{A}$, $\sum_{B \in \partition{B}} | P_{jk}(B) - \Lambda_{jk}(B) | < \varepsilon$. We can conclude that $\Lambda_{jk}$ is the strict additive transform of $P_{jk}$ on $J$. Since $j$ and $k \neq j$ are arbitrary this also establishes that $\Lambda$ is the strict additive transform of $P - \one$ since $\Lambda_{jj}(A) = - \sum_{k \neq j} \Lambda_{jk}$ and $P_{jj}(A) - 1 = - \sum_{k \neq j} P_{jk}(A)$.
\end{proof}
In fact, only a right continuity property rather than an upper continuity property at $\emptyset$ is used for $P - \one$ in the proof above.

The importance of Theorem~\ref{theorem:Lambda_vs_P} comes from Theorem~\ref{theorem:additive_multiplicative} which then states that the strict multiplicative transform of $P$ exists and corresponds to the strict multiplicative transform of $\one + \Lambda$. This means that the multiplicative transform $\prodi \d P$ of $P$ equals the multiplicative transform of $\one + \Lambda$, or in other terms, for any $t > 0$,
\begin{equation}
\Prodi_0^t(\one + \Lambda(\d s)) = \lim \prod_{i=1}^m P(t_{i-1}, t_i)
\end{equation}
where the limit is over refinements of $(0,t]$ which was the desired result.

As argued in the proof of Theorem~\ref{theorem:Lambda_vs_P}, the assumption $\|P - \one\|_{(1)} < \infty$ implies the more standard assumption in multi-state settings, namely that $\Lambda_{jk}(J) < \infty$ for all $j$ and $k \neq j$.
In the Markov case, $\|P - \one\|_{(1)} < \infty$ is, however, implied by $\Lambda_{jk}(J) < \infty$ for all $j$ and $k \neq j$.
The convention that $P_{jk}((s,t]) = 0$ for all $t$ when $p_j(s) = 0$ would have to be abandoned for something that agrees with multiplicativity of $P$. The convention that $P_{jk}(A) = \indic{j=k}$ for $A$ with $p_j(s) = 0$ for all $s \in A$ suffices. 

\begin{prop}
In the Markov case, assume $P - \one$ is upper continuous at~$\emptyset$ and that $\|\Lambda\|_{(1)} < \infty$. Then $P(A) = \prodi_A \d (\one + \Lambda)$ for all $A \in \J$ and in particular
\begin{equation}
    P(s,t) = P((s,t]) = \Prodi_{(s,t]} \d (\one + \Lambda) = \Prodi_s^t(\one + \Lambda(\d u))
\end{equation}
for all $(s,t] \in \J$, and $\|P-\one\|_{(1)} < \infty$.
\end{prop}
\begin{proof}
Since $\Lambda$ is additive and of bounded variation according to the assumption, $\one + \Lambda$ has a strict multiplicative transform, here denoted by $A \mapsto \prodi_A \d (\one + \Lambda)$, by Theorem~\ref{theorem:additive_multiplicative}.
As in the proof of Theorem~\ref{theorem:Lambda_vs_P}, the conclusion that $\Lambda_{jk}$ is the strict additive transform of $P_{jk}$ on $[t_0, t_1) \subseteq J_{j+} := \{s \in J : p_j(s) > 0 \textup{ and } p_j(s-) > 0 \}$ remains valid under the assumption of $\Lambda_{jk}(J) < \infty$ for all $j$ and $k \neq j$ since this is enough to ensure $\|Q_{jk}\|_{(1)} < \infty$ by the inequalities $\|Q_{jk}\|_{(1)} \leq \sum_{h \neq j} F_{jh}(J) \leq \sum_{h \neq j} \Lambda_{jh}(J)$. And this continues to lead to $\Lambda$ being the strict additive transform of $P - \one$ on intervals $[t_0, t_1)$ that are either in $J_{j+}$ or in the interior of $J_{j0} := \{s \in J : p_j(s) = 0 \textup{ or } p_j(s-) = 0 \}$ for all $j$. By Theorem~\ref{theorem:additive_multiplicative}, $P(A) = \prodi_A \d(\one + \Lambda)$ for subintervals of such $[t_0,t_1)$ since $P$ is multiplicative and therefore its own strict multiplicative transform on such $[t_0,t_1)$.
A partition as in the proof of Theorem~\ref{theorem:Lambda_vs_P} can be made such that $(0,\tau] = U_1 \cup \{u_1\} \cup \dots \cup \{u_{m-1}\} \cup U_m \cup \{u_m\}$ with open intervals $U_i$ for which either $U_i \subseteq J_{j+}$ or $U_i \subseteq J_{j0}$ for all $j$. 
Upper continuity at $\emptyset$ of $P-\one$ and multiplicativity of $P$ reveal that $P((s,t)) = \lim_{u \downarrow s}P((s,u)) P([u,t)) = \lim_{u \downarrow s} P([u,t))$ such that 
\begin{equation}
P((s,t)) = \lim_{u \downarrow s} P([u,t)) = \lim_{u \downarrow s} \Prodi_{[u,t)} \d (\one +  \Lambda) = \Prodi_{(s,t)} \d(\one + \Lambda)
\end{equation}
for any $(s,t) \subseteq U_i$ for some $i$ where upper continuity of the multiplicative transform is also used. 
This reveals $P(A) = \prodi_A \d (\one + \Lambda)$ for any subinterval $A \subseteq U_i$ for some $i$. Trivially, $P([u_i, u_i]) = \one + \Lambda([u_i,u_i]) = \prodi_{[u_i,u_i]} \d(\one + \Lambda)$.
Multiplicativity now allows us to glue together any interval such that $P(A) = \prodi_A \d (\one + \Lambda)$ for all $A \in \J$ and thus $P$ is the strict additive transform of $\Lambda$. The argument at the end of the proof of Theorem~\ref{theorem:additive_multiplicative} reveals that $\|P- \one\|_{(1)} \leq \exp(\|\Lambda\|_{(1)}) \|\Lambda\|_{(1)} < \infty$ in this case.
\end{proof}
Again, what is really used in the proof above is a right continuity property rather than an upper continuity property of $P - \one$ at $\emptyset$. This makes the proposition very similar to Theorem~15 of \cite{Gill1990}. The contribution of the proposition is, as mentioned, that $\|\Lambda\|_{(1)} < \infty$ implies $\|P - \one\|_{(1)}$ in this setting. 

\section*{Acknowledgements}
The author would like to thank Erik Thorlund Parner and Jan Pedersen for discussions and comments on drafts of this paper. This research is supported by the Novo Nordisk Foundation, grant NNF17OC0028276.

\appendix

\section{Useful technical results}
Consider the multi-state setting of sections~\ref{sec:multi_state} and~\ref{sec:main_result} and let $J_{j+} = \{s \in (0,\infty) \given p_j(s) > 0 \textup{ and } p_j(s-) > 0\}$ for each $j \in \{1, \dots, d\}$. 

\begin{lemma} \label{lemma:pos_set}
Consider $j \in \{1, \dots, d\}$. For any closed interval $[s,t] \subseteq J_{j+}$ an $\varepsilon > 0$ exists such that $p_j(u) \geq \varepsilon$ and $p_j(u-) \geq \varepsilon$ for all $u \in [s,t]$.
\end{lemma}

\begin{proof}
Assume not for an interval $[s,t]$. Then we can find a sequence $(u_n)_{n \in \N} \subseteq [s,t]$ such that $p_j(u_n) < n^{-1}$ or $p_j(u_n-) < n^{-1}$. Since $[s,t]$ is bounded and closed, a monotone and thus convergent subsequence $(u_{n_k})_{k \in \N}$ exists according to the Bolzano--Weierstrass Theorem with a limit $u = \lim_{k \to \infty} u_{n_k} \in [s,t]$.
If $u_{n_k} = u$ from a certain point then $p_j(u-) = p_j(u_{n_k}-) < n_k^{-1}$ or $p_j(u) = p_j(u_{n_k}) < n_k^{-1}$ for any large $k$. This implies $p_j(u-) = 0$ or $p_j(u) = 0$. If $u_{n_k} \neq u$ for all $k$ for the monotone, convergent sequence $(u_{n_k})$ then $\lim_{k \to \infty}p_j(u_{n_k}-) = \lim_{k \to \infty}p_j(u_{n_k})$ since $p_j$ has limits from either direction and this limit is 0 due to the properties of the sequence. If $u_{n_k}$ is increasing towards $u$ then $p_j(u-) = \lim_{k \to \infty} p_j(u_{n_k}) = 0$. Similarly, if $u_{n_k}$ is decreasing towards $u$ then $p_j(u) = \lim_{k \to \infty} p_j(u_{n_k}) = 0$.
The conclusion is that we either have $p_j(u) = 0$ or $p_j(u-) = 0$ for $u \in [s,t]$ such that $[s,t]$ cannot be a subset of $J_{j+}$. This proves the lemma by contraposition.
\end{proof}

\begin{lemma} \label{lemma:pj_inequality}
For any $j \in \{1, \dots, d\}$,
\begin{equation}
    p_j(t) \geq p_j(s) \Prodi_s^t(1-\Lambda_{j\bullet}(\d u))
\end{equation}
for $t \geq s$ where $\Lambda_{j\bullet} = \sum_{k \neq j} \Lambda_{jk}$.
\end{lemma}
\begin{proof}
The claim is trivial if $p_j(s) = 0$, so assume $p_j(s) > 0$. 
The right-continuity of $p_j$ now implies the existence of an interval $(s,u] \subseteq J_{j+}$. Let $\tilde \tau = \sup \{u \in (0,\infty) : (s,u] \subseteq J_{j+}\}$, then $\tilde \tau \notin J_{j+}$ since $J_{j+}$ and the set that the supremum is taken over are open. Suppose $\tilde \tau < \infty$, then $p_j(\tilde \tau-) = 0$ or $p_j(\tilde \tau-) > 0$ and $p_j(\tilde \tau) = 0$. In the latter case, $\Lambda_{j\bullet}([\tilde \tau, \tilde \tau]) = p_j(\tilde \tau-)^{-1} \sum_{k \neq j}(F_{jk}(\tilde \tau) - F_{jk}(\tilde \tau-)) = 1$ and so $\prodi_s^{\tilde \tau}(1 - \Lambda_{j\bullet}(\d u)) = 0$ in which case the inequality is trivial for $t \geq \tilde \tau$ by multiplicativity of the product integral. In the former case, if we can prove that the inequality holds for all $t < \tilde \tau$, then the inequality continues to hold in the limit from the left, $0 = p_j(\tilde \tau-) \geq p(s) \prodi_s^{\tilde \tau-}(1 - \Lambda_{j\bullet}(\d u))$ implying $\prodi_s^{\tilde \tau-}(1 - \Lambda_{j\bullet}(\d u)) = 0$ since $p_j(s) > 0$ by assumption. This in turn implies $\prodi_s^{t}(1 - \Lambda_{j\bullet}(\d u)) = 0$ for $t \geq \tilde \tau$ by multiplicativity of the product integral in which case the inequality is trivial. So we only need to consider $t < \tilde \tau$, in particular with $(s,t] \subseteq J_{j+}$, and we may also assume $\prodi_s^t(1- \Lambda_{j\bullet}(\d u)) > 0$ since the inequality is trivial otherwise. Under this last assumption we have $\Lambda_{j\bullet}((s,t]) < \infty$ and so $F_{jk}((s,t]) < \infty$ for all $k \neq j$. 
Define an interval function by $P_{j \bullet} = \sum_{k \neq j} P_{jk}$.
It is worth noting that
\begin{equation}
\begin{aligned}
    p_j(u) &= p_j(s) + \sum_{k \neq j}Q_{kj}((s,u]) - \sum_{k \neq j}Q_{jk}((s,u]) \\
    & \geq p_j(s)(1- P_{j\bullet}((s,u]))
\end{aligned}
\end{equation}
since $P_{j \bullet}((s,u]) = p_j(s)^{-1} \sum_{k \neq j} Q_{jk}((s,u])$.
Iterating on this inequality leads to
\begin{equation} \label{eq:ineq_iter}
    p_j(t) \geq p_j(s) \prod_{i=1}^m (1-P_{j\bullet}((t_{i-1},t_i]))
\end{equation}
for any partition at points $s = t_0 < t_1 < \dots < t_m = t$. 
In particular, the inequality holds in the limit over refinements when it exists.
In similarity to the result of Lemma~\ref{lemma:pos_set}, we have $p_j(u) \geq \varepsilon$ for all $u \in (s,t]$ for some $\varepsilon > 0$. 
By Theorem~\ref{theorem:integral} with the regulated function $u \mapsto p_j(u-)$ and the interval function $\sum_{k \neq j} Q_{jk}$ with variation and strict additive transform given by $\sum_{k \neq j} F_{jk}$, we obtain that $\Lambda_{j\bullet}$ is the strict additive transform of $P_{j\bullet}$ on $(s,t]$ and then by Theorem~\ref{theorem:additive_multiplicative} that the limit over refinements of $(s,t]$, $\lim \prod_{i=1}^m (1-P_{j\bullet}((t_{i-1},t_i]))$ is $\prodi_s^t(1- \Lambda_{j\bullet}(\d u))$. 
The inequality $p_j(t) \geq p_j(s) \prodi_s^t(1- \Lambda_{j\bullet}(\d u))$ now follows from \eqref{eq:ineq_iter}.
\end{proof}

\begin{lemma} \label{lemma:Lambda_extremes}
If $[s,t) \subseteq J_{j+}$ with $t \notin J_{j+}$ for $t < \infty$ then $\sum_{k \neq j}\Lambda_{jk}((s,t)) = \infty$ if $p_j(t-) = 0$ or $\sum_{k \neq j}\Lambda_{jk}([t,t]) = 1$ if $p_j(t-) > 0$.
\end{lemma}
\begin{proof}
Lemma~\ref{lemma:pj_inequality} and its proof reveal how $\prodi_s^{t-} (1-\sum_{k \neq j} \Lambda_{jk}(\d u)) = 0$ if $p_j(t-) = 0$ or $\sum_{k \neq j} \Lambda_{jk}([t,t]) = 1$ if $p_j(t-) > 0$. If $\prodi_s^{t-} (1-\sum_{k \neq j} \Lambda_{jk}(\d u)) = 0$ with $\sum_{k \neq j} \Lambda_{jk}([u,u]) < 1$ for all $u \in (s,t)$ as in this case since $[s,t) \subseteq J_{j+}$, then $\sum_{k \neq j}\Lambda_{jk}((s,t)) = \infty$, which can be seen by appealing to the decomposition of $\Lambda_{jk}$ into its continuous and discrete parts as in Definition~4 of \cite{Gill1990}. Essentially, we have $\prod_{u \in (s,t)}(1-\sum_{k \neq j}\Lambda_{jk}([u,u])) = 0$ or $\exp(-\sum_{k \neq j}\Lambda_{jk}((s,t))) = 0$, and either scenario implies ${\sum_{k \neq j}\Lambda_{jk}((s,t)) = \infty}$. 
This proves the lemma.
\end{proof}

\end{document}